\documentclass[lm,a4paper,twoside,fleqn,leqno,sc,color]{pre} 

\usepackage[applemac]{inputenc}
\usepackage{math}

\newtheorem{ml}[thm]{Main Lemma}
\newtheorem*{ga}{General Assumption}
\newtheorem*{ssl}{Siegel-Sternberg linearization theorem}
\newtheorem*{cl}{Counting Lemma}

\let\ab\ang
\let\cle\curlyeqprec
\let\vle\eqslantless
\let\down\shortdownarrow
\let\pcmp\leftslice
\let\msub\sqsubset

\def\Td{\dot T}
\def\Md{\dot M}
\def\mb{\breve m}
\def\mt{\tilde m}
\def\Mb{\breve M}
\def\xs{\acute x}
\def\mp{\acute m}
\def\vq{\acute s}
\def\rn{^{1/n}}
\def\nb{\bar n}
\def\hd{\hslash}

\def\rbr{r,\bots,r}

\def\fdb{\textsc{fdb}\xspace}
\def\asm{\textsc{asm}\xspace}

\def\qtt#1{›#1‹}
\def\lox#1{\vph{x}_{\scriptstyle#1}}

\begin{document}

\title     {On the Siegel-Sternberg linearization theorem}
\author    {Jürgen Pöschel}
\date      {Version 4.1, Februar 2017
			}
\address   {Institut für Analysis, Dynamik und Optimierung\\
            Universität Stuttgart, Pfaffenwaldring 57, D-70569 Stuttgart\\poschel@mac.com}
\annotation{In memory of Tommy 1999--2017}
\maketitle

\begin{abstract}
We establish a general version of the Siegel-Sternberg linearization theorem for ultradiffentiable maps which includes the analytic case, the smooth case and the Gevrey case. It may be regarded as a small divisior theorem without small divisor conditions. Along the way we give an exact characterization of those classes of ultradifferentiable maps which are closed under composition, and reprove regularity results for solutions of ode's and pde's. 
This will open up new directions in \textsc{kam}-theory and other applications of ultradifferentiable functions.
\end{abstract}

We consider the problem of linearizing a map~$g$ in the neighbourhood of a fixed point. Placing this fixed point at the origin we write
\[
  g = \Lm+\gh,
\] 
where $\Lm$ denotes its linear part and $\gh$ comprises its nonlinear terms. We then look for a diffeomorphism
$ \ph = \id+\phh $
around the origin such that
\[
  \ph\inv \cmp g\cmp \ph = \Lm.
\]
It is well known that any solution to this problem depends on the eigenvalues of its linearization. Let $g$ be a map in \m{s}-space, and let $\lm_1,\bots,\lm_s$ be the eigenvalues of~$\Lm$.
Within the category of formal power series there is always a formal solution to this problem, if the infinitely many \emph{nonresonance conditions}
\[
  \lm^k-\lm_i \ne 0, 
  \qq 
  \n{k}\ge2, \q 1\le i\le s,
\]
are satisfied, where 
\mmx
$\lm^k=\lm_1^{k_1}\cbots \lm_s^{k_s}$ and $\n{k}=k_1+\bots+k_s$
\[
  \lm^k=\lm_1^{k_1}\cbots \lm_s^{k_s},
  \qq
  \n{k}=k_1+\bots+k_s
\]
for $k\in\set{0,1,\bots}^s$.
Convergence of these formal solutions, however, can be established only if the map is analytic and certain strong \emph{small divisor conditions} are satisfied, such as
\[
  \n*{\lm^k-\lm_i} \ge \frac{\gm}{\n{k}^\ta}, 
  \qq 
  \n{k}\ge2, \q 1\le i\le s,
\]
with some $\gm>0$ and large $\ta$. This is the celebrated result of Siegel~\cite{Sie-42,Sie-52}, who was the first to overcome those small divisor difficulties.
Later, those small divisor conditions were considerably relaxed by Bruno and Rüssmann~\cite{Bru-72,Rus-80}. On the other hand, it is also well known that bad small divisors destroy analyticity.

A few years later, Sternberg~\cite{Stb-58} established the same result within the smooth category without any small divisor conditions. More precisely, if $g$ is real, smooth and nonresonant, then $g$ can be linearized by a smooth diffeomorphism around the fixed point. 
The construction proceeds in two steps. First, the nonlinearity is removed up to a flat term using nonresonance. Then, the latter is removed using the hyperbolicity of $\Lm$, which is a consequence of nonresonance and reality. 
So it appears »that small divisors are invisible in the smooth category«~\cite{Sto-13}.

The purpose of this paper is to prove that this is not the case.
The smooth category by itself is just too indifferent to discern small divisiors. But looking more closely in terms of classes of ultradifferentiable functions one can clearly quantify the effect of near resonances no matter how small they get how fast. The results of Siegel and Sternberg are then two instances of one and the same general theorem, as are all other results in this category. 

Ultradifferentiable functions form subclasses of smooth functions that are described by growth restrictions on their derivatives. Those restrictions are given in terms of a sequence of positive real numbers that serve as weights for those derivatives.
More specifically, if
\mms $
  m=(m_n)_{n\ge1}
$ 
is a sequence of positive numbers, then a smooth function $f$ on some \m{s}-dim\-en\-sional domain is said to be of class $E^m$, if for any point in the domain of $f$ there is a neighborhood~$U$ and a positive number $r>0$ such that
\[
  \sup_{x\in U}\, \sup_{n\ge1} \frac{1}{n\fac}\frac{\nn*{D^n\f(x)}}{m_n}\,r^n < \iny.
\]
Those classes are also called local Denjoy-Carleman classes of Roumieu type and are often denoted by $\mathscr{E}^{\set{m}}$ or ${C}{\set{m}}$.
Best known among those are the Gevrey classes $G^s$ defined in terms of the weights
\[
  m = (n\fac^{s-1}), \qq s>1.
\] 
Indeed, a version of the Sternberg theorem was recently proven by Sto\-lo\-vitch~\cite{Sto-13} for Gevrey maps.

To state the general result we measure the size of near-resonances in terms of their \emph{nonresonance function} $\Om$ defined by
\[
  \Om(q) = \max_{2\le\n{k}\le q} \max_{1\le i\le s} \n*{\lm^k-\lm_i}\inv, \qq q\ge2.
\]
We say that a weight $w$ \emph{dominates a nonresonance function} $\Om$, if there exists a constant~$a$ such that
\[
  \sum_{\nu\le\log_2\n{k}} \frac{\log \Om(2^{\nu+1})}{2^\nu} \le a + \frac{\log w_k}{\n{k}},
  \qq
  \n{k}\ge2.
\]
The point is that \emph{any} nonresonant function can be dominated by an appropriate weight.
Hence, the following theorem may be regarded as a small divisior theorem without explicit small divisior conditions.

\begin{ssl}
Consider a smooth map~$g$ of class $E^m$ in a neighbourhood of a fixed point. If its linearization at the fixed point is nonresonant, then the map can be linearized by a local diffeomorphism~$\ph$ of class $E^{m\star w}$, where
$w$ is any weight dominating the associated nonresonance function $\Om$ such that $m\star w$ is log-convex and strongly non-analytic.
\end{ssl}

Here, $m\star w = (w_km_k)_{k\ge0}$ is log-convex, if the logs of the weights are convex with respect to~$k$. Strongly non-analytic weights are defined in section~~\ref{part-1}. Both conditions are nothing more than weak growth conditions.

This theorem comprises all versions of the Sternberg linearization theorem established so far. We discuss the most relevant cases. 

\head{No small divisors}
In this case, $\Om$ is bounded. 
This amounts to the classical theorem of Poincaré \cite{Poi-79} and is a particularly simple instance of the next case.

\head{Good small divisors}
If the eigenvalues of $\Lm$ satisfy small divisor conditions of Siegel or Bruno-Rüssmann type, then
\[
  \sum_{\nu\ge0} \frac{\log \Om(2^{\nu+1})}{2^\nu} < \iny.
\]
Indeed, this is the general \emph{definition} of admissible small divisors for convergent majorant techniques as introduced by Bruno~\cite{Bru-72}. In this case, we simply choose 
the constant weight
\mms $w=(1)$. 
So the normalizing transformation $\ph$ is of the \emph{same class} $E^m$ as the map~$g$. This applies to the  analytic category~$C^\om$ -- which amounts to the classical results of Siegel~\cite{Sie-42,Sie-52} --, the Gevrey category~$G^s$ -- see Stolovitch~\cite{Sto-13} --, and \emph{any other} \fdb space $E^m$ as defined in Lemma~\ref{prop-fdb}.

\head{Gevrey small divisors}
This amounts to the existence of a $\dl>0$, so that 
\[
  \sum_{1\le\nu\le\log_2\n{k}} \frac{\log \Om(2^{\nu+1})}{2^\nu} \le a + \dl \log\n{k}
\]
for almost all $k$.
In this case, we can choose
$
  w = (k\fac^\dl)
$.
So if $g$ is of Gevrey class $G^s$, then $\ph$ is of Gevrey class $G^{s+\dl}$ -- see again~\cite{Sto-13}. But the same loss of regularity is observed in any other \fdb space $E^m$.

\head{Arbitrarily small divisor}
The theorem also applies to the case where no small divisor estimate and no smoothness class are given at all.
\emph{Any} smooth map $g$ is of \emph{some} class~$E^m$, since we only need to choose an appropriate weight $m$ in dependence on the growth of the derivatives of~$g$. Moreover, there always is \emph{some} weight $w$ dominating the associated resonance function~$\Om$. Increasing $w$ if necessary, $m\star w$ is log-convex and strongly non-analytic. Hence, the theorem applies also in this case and amounts to the general Sternberg theorem with additional quantitative information.

\paragraph{Outline}
An indispensable prerequisite for doing analysis within spaces of ultradifferentiable functions is their stability with respect to composition. Partial results are well known and rely on the Faà di Bruno formula for higher derivatives of composite functions and the assumption that derivation is well behaved. The latter, however, amounts to a severe growth restriction on the weights~$w$. The essential step is to completely remove the latter restriction and to give an exact description of those spaces. The proof is also much simpler and works by considering formal power series only.
As an application of this approach we reprove regularity results for solutions of ode's and pde's 
without employing tedious estimates.

The proof of the Siegel-Sternberg theorem then consists of two parts. First, a small divisor problem is solved to linearize the map $g$ up to a flat term. But working within the category of ultradifferentiable functions it is not necessary to use any bounds on those divsisors. It suffices to keep control of their effect and proliferation.
Subsequently, hyperbolicity is used to remove the flat term. Here, we transfer the classical approach to the proper $E^m$ classes using heavily their closedness under composition und flows and also using a version of the Whitney extension theorem.

\paragraph{Acknowledgement}
It is a pleasure to thank Gerhard Schindl for carefully reading a preliminary version of this manuscript and pointing out some errors.

\section{Ultradifferentiable functions and maps}

First consider complex valued functions.
With any smooth complex valued function~$f$ defined in a neighborhood of a point~$a$ in real \m{s}-space we associate its formal Taylor series expansion at~$a$,
\[
  T_af \defeq \sum_{k\in\Lm} f_k(a)x^k,
\]
where as usual 
\[
  f_k 
  \defeq \frac{1}{k\fac} \del^k\f
  \defeq \frac{1}{k_1!\cbots k_s!}\,\del_{x_1}^{k_1\vph{k_s}}\cbots\del_{x_s}^{k_s} f,
  \qq
  x^k = x_1^{k_1}\cbots x_s^{k_s}
\]
for \m{s}-dimensional multiindices $k$ in $\Lm = \set{0,1,\bots}^s$.
As constant terms won't play a role in our considerations, we also let
\[
  \Td_a f \defeq T_af-f_0 = \sum_{k\ne0} f_k(a)x^k.
\]

A \emph{weight} is any map
\mmx
$m\maps \Lm \to (0,\iny)$.
\[
  m\maps \Lm \to (0,\iny).
\]
The \emph{weighted Taylor series} of $f$ with and without constant term are then defined as
\[
  M^m_a\f \defeq \sum_{k\in\Lm} \frac{\n{f_k(a)}}{m_k}x^k,
  \qq
  \Md^m_a\f \defeq M^m_a\f-f_0,
\]
respectively, and we set
\[
  \nn{f}^m_{a,r} 
  \defeq \sum_{k\ne0} \frac{\n{f_k(a)}}{m_k} r^{\n{k}}
  = \Md^m_af(\rbr).
\]
Obviously, $f$ is of class $E^m$ if and only if any point in the domain of $f$ has a neighborhood~$U$ and a positive number $r>0$ such that
\[
  \nn{f}^m_{U,r} \defeq \sup_{a\in U} \nn{f}^m_{a,r} < \iny.
\]
Note that we do not take into account a constant term, so these are only semi-norms.

Here are some standard examples.
In the one-dimensional case the constant sequence $m=(1)_{n\ge0}$ 
defines the class of analytic functions on open subsets of the real line,
\mms$
  C^\om = E^{(1)}
$.
More generally, as shown in Appendix~\ref{app-1},
\[
  C^\om = E^m \eequi 0<\inf m_n\rn \le \sup m_n\rn < \iny.
\]
For $m=(n\fac^{s-1})_{n\ge1}$ we obtain the Gevrey spaces~\cite{Gev-18} 
\[
  G^{s} = E^{(n\fac^{s-1})}, \qq s>1,
\]
well known in pde theory.

These examples naturally extend to the multi-dimensional case. Here, one usually considers weights as functions of $\n{k}=k_1+\bots+k_s$ rather than~$k$. For instance,
\[
  G^{s} = E^{(k\fac^{s-1})} = E^{(\n{k}\fac^{s-1})}
\]
by standard inequalities for the factorial. But with true multi-dimensional weights one may also consider functions with anisotropic differentiability properties -- see for example \cite{CaG-09} and section~\ref{flows}. 

We also need to consider smooth maps from \m{s}-space into some \m{\vq}-space. For $f=(f_1,\bots,f_{\vq})$ we set
\[
  \nn{f}^m_{a,r} \defeq \max_{1\le i\le\vq} \nn{f_i}^m_{a,r}.
\]
The Taylor coefficients $f_k$ of $f$ are \m{\vq}-vectors. Defining
\[
  \Md^m_a\f \defeq \sum_{k\ne0} \frac{\ab{f_k(a)}}{m_k}x^k,
  \qq
  \ab{f_k} \defeq (\n*{f_{k,1}},\bots,\n*{f_{k,\vq}}),
\]
and denoting the usual sup-norm by $\nny\cd$ we have
\[
  \nn{f}^m_{a,r} = \nn*{\Md^m_af(\rbr)}_\iny.
\]
In either case, $f$ is of class~$E^m$ if and only if $\Md^m_a\f$ is analytic on a \m{s}-dimen\-sional polydisc with a radius that can be chosen locally constant. 

\section{Basic properties and assumptions}  \label{basic}

For the time being we focus on one-dimensional weights. Multi-dimensional weights will be considered again in section~\ref{comps}.

Certain properties of the spaces $E^m$ are more directly connected with the sequence $M=(M_n)_{n\ge1}$ of the associated weights
\[
  M_n \defeq n\fac m_n, \qq n\ge1,
\]
controlling the derivatives $f^{(n)}$ rather than the Taylor coefficients~$f_n$.
If
\[
  A \defeq \liminf M_n\rn < \iny,
\]
then $E^m$ is a proper subspace of $C^\om$ and not closed under composition of maps -- see Appendix~\ref{app-1}. Hence we will assume that $A=\iny$.
In this case one always has
\[
  E^m = E^{\mb},
\]
where $\mb$ is characerized by the fact that $\Mb$ is the largest log-convex minorant of~$M$. 
That is, $\Mb_n^2 \le \Mb_{n-1}\Mb_{n+1}$, which is tantamount to $\Mb_{n}/\Mb_{n-1}$ forming an increasing sequence \cite{Ban-46,Man-52}.
---
From now on we therefore make the following

\begin{ga}
The weights $m=(m_n)_{n\ge1}$ are \qtt{weakly log-convex}:
\[
  m_n = \frac{M_n}{n\fac}
\]
with a log-convex sequence $M=(M_n)_{n\ge1}$ so that $M_n = \mu_1\cbots\mu_n$ with an increasing sequence
\[
  0 < \mu_1 \le \mu_2 \le \bots\en.
\]
As a consequence, $M_n\rn$ is increasing and
\mmx
$M_n\rn\up\iny$ if and only if $\mu_n\up\iny$.
\[
  M_n\rn \up \iny \eequi \mu_n \up \iny.
\]
\end{ga}

Under this assumption a weight $m$ is always \qtt{weakly submultiplicative}: we have
$M_kM_l\le M_{k+l}$ and thus, by the binomial formula,
\[
  m_km_l = \frac{M_kM_l}{k\fac\,l\fac}
  \le \frac{(k+l)\fac}{k\fac\,l\fac} \frac{M_{k+l}}{(k+l)\fac}
  \le 2^{k+l} m_{k+l},
  \qq
  k,l\ge1.
\]
As a consequence, $E^m$ is always an algebra. But note that $m$ is not necessarily almost submultiplicative as defined in Lemma~\ref{prop-inv}.

Another important consequence \opt{of this assumption}
is the existence of so called characteristic \m{E^m}-functions. The following lemma is well known, as is its proof \cite{Ban-46,Jaf-15}. We state it for functions on an interval.

\begin{lem} \label{char-func}
Under the general assumption the space $E^m$ contains for any given point $o$ in the interval under consideration a function $\eta$ such that $\eta(o)=0$ and
\[
  \eta_n(o) = \j^{n-1}s_n, \qq s_n \ge m_n , \qq n\ge1.
\]
\end{lem}

\begin{proof}  \let\rh=\sg
Set $T_n = \mu_n^n/M_n$. We then have
\[
  T_n = \sup_{k>0} \frac{\mu_n^k}{M_k},
\]
as $\mu_n/\mu_k\ge 1$ for $k\le n$ and $\mu_n/\mu_k\le 1$ for $k\ge n$.
Now assume for simplicity that $o$ is the origin on the real line and define $\rh_\nu$ by
\[
  \rh_\nu(x) = \frac{\e^{\j\mu_\nu x}}{T_\nu}, \qq \nu\ge1.
\]
Its \m{n}-th derivative is uniformly bounded by
\[
  \nn*{\rh_\nu^{(n)}} \le \frac{\mu_\nu^n}{T_\nu} \le M_n, \qq n\ge1.
\]
Hence, if we define $\ph$ by
\[
  \ph(x) = \sum_{\nu\ge1} 2^{-\nu}\rh_\nu(x) ,
\]
then $\nn*{\ph^{(n)}}\le M_n$ and $\nn*{\ph_n}\le M_n/n\fac = m_n$ for all $n\ge1$. Hence $\ph$ is in~$E^m$, and its Taylor coefficients at zero are
\[
  \ph_n = \j^n s_n,  \qq n\ge 1,
\]
with
\[
  s_n
  =   \frac{1}{n\fac} \sum_{\nu\ge1} \frac{1}{2^\nu} \frac{\mu_\nu^n}{T_\nu}
  \ge \frac{1}{n\fac} \frac{1}{2^n} \frac{\mu_n^n}{T_n}
  =   \frac{1}{2^n} \frac{M_n}{n\fac}
  =   \frac{m_n}{2^n}.
\]
So the rescaled function $\eta=-\j\ph\cmp 2-\eta_0$ with a suitable constant $\eta_0$ 
has all the required properties.
\end{proof}

We note that the proof does not make use of the assumption that $\mu_n\to\iny$. But if not, then $E^m$ consists of analytic functions, and the result is trivial.

From the existence of characteristic functions it follows that
\[
  E^{\mt} \subset E^m \eequi \mt\msub m, 
\]
where the latter stands for $\mt_n\le\lm^n m_n$ for all $n\ge1$ with some $\lm\ge1$.
Consequently,
\[
  E^{\mt} = E^{m} \eequi \mt \asymp m \defequi \mt\msub m \land m\msub\mt.
\]
Obviously, $\asymp$ is an equivalence relation among weights, identifying all weights which define the same \m{E}-space.

\section{Properties of weights}

All of the following properties pertain to the equivalence class of a weight, thus are properties of the associated spaces~$E^m$. The corresponding analytical properties will be discussed later. 

We will use Stirling's inequality in the form
\[
  \frac{n}{\e} \le n\fac\rn \le \frac{2n}{\e},
  \qq n\ge2.
\]
For instance, as $M_n\rn$ is increasing by the general assumption it follows that
\[
  •mkml
  \frac{{m_k}^{1/k}}{{m_l}^{1/l}}
  = \frac{l\fac^{1/l}}{k\fac^{1/k}} \frac{{M_k}^{1/k}}{{M_l}^{1/l}}
  \le \frac{l\fac^{1/l}}{k\fac^{1/k}}
  \le \frac{2l}{k},
  \qq
  1\le k\le l.
\]
We will use this estimate in the next proof.

\begin{lem} \label{prop-inv}
The following two properties are equivalent.
\enum[z]
\item•1
$m$ is \qtt{almost increasing}:\en there is a $\lm\ge1$ such that
\[
  m_k^{1/k} \le \lm m_l^{1/l},
  \qq
  1\le k\le l.
\]
\item•2
$m$ is \qtt{almost submultiplicative} or \qtt{{\asm}}:\en there is a $\lm\ge1$ such that 
\[
  m_{k_1}\cbots m_{k_r} \le \lm^k m_k, \qq k=k_1+\bots+k_r,
\]
for any choice of\/ $r\ge1$ and $k_1,\bots,k_r\ge1$.
\endenum
\end{lem}

\begin{proof}
If $m$ is almost increasing and $k=k_1+\bots+k_r$, then
\[
  m_{k_i} \le \lm^{k_i} m_{k}^{k_i/k}, \qq 1\le i\le r.
\]
Taking the product over $1\le i\le r$ gives the second property. 
Conversely, if $m$ is \asm, then in particular
\mms $
  m_k^n \le \lm^{nk} m_{nk}
$
and thus
\[
  m_k^{1/k} \le \lm m_{nk}^{1/nk}, \qq n,k\ge1.
\]
Given $1\le k\le l$ and choosing $n\ge1$ so that $nk\le l\le nk+k$ we further conclude with~\eqref{mkml} that
\[
  m_{nk}^{1/nk} \le 4m_l^{1/l}.
\]
These two estimates together show that $m$ is almost increasing.
\end{proof}

In the next lemma, \qtt{{\fdb}} is short for \qtt{Faà di Bruno}. The property thus named is motivated by the composition rule for formal power series -- see the Main Lemma~\ref{comp-wm} -- and reflects the higher order chain rule named after Faà di Bruno~\cite{FdB-57}. The term was coined in~\cite{RaS-14}.

\begin{lem}  \label{prop-fdb}
Each of the following properties implies the next one.
\enum[z]
\item•1
$m$ is \qtt{log-convex}:\q $m_n^2\le m_{n-1}m_{n+1}$ for all $n\ge2$, or equivalently, 
\[
  m_n=\al_1\al_2\cbots\al_n, \qq 0<\al_1\le\al_2\le\bots\en .
\]
\item•2
$m$ is \qtt{block-convex}:\en 
\[
  \max_{n\le 2^\nu} \al_n \le \min_{n>2^\nu} \al_n, \qq \nu\ge0.
\]
\item•9
$m$ is \qtt{strongly submultiplicative}:\en there is a $\lm\ge1$ such that
\[
  m_km_l \le \lm m_{k+l-1}, \qq k,l\ge1.
\]
\item•4
$m$ is \qtt{strictly {\fdb}}:\en there is a $\lm\ge1$ so that for all $r\ge1$ and $k_1,\bots,k_r\ne0$,
\[
  m_rm_{k_1}\cbots m_{k_r} \le \lm^r m_{k_1+\bots+k_r}.
\]
\item•5
$m$ is \qtt{{\fdb}}:\en there is a $\lm\ge1$ so that for all $r\ge1$ and $k_1,\bots,k_r\ne0$,
\[
  m_rm_{k_1}\cbots m_{k_r} \le \lm^k m_k, \qq k=k_1+\bots+k_r.
\]
\endenum
\end{lem}

Note that~\ref{9} and~\ref{4} hold with $\lm=1$ by passing to an equivalent weight. This is not true for~\ref{5}. Indeed, we do not know whether \ref{4} and \ref{5} are equivalent or not.
---
The block-convex property seems to be a new concept and offers a more flexible way to construct strictly \fdb weights -- see Example~\ref{fdb-not-log}.

\begin{proof}
\zg1\imp2.
This is obvious.

\zg2\imp9.
Dividing all $m_n$ by $m_1$ we may assume that $m_1=1$. 
Now, for given $1\le l\le k$ we fix $\nu\ge0$ so that $2^\nu\le k<2^{\nu+1}$. If $l\le 2^\nu$, then by hypnosis
\[
  \al_2\cbots\al_{l} \le \al_{k+1}\cbots\al_{k+l-1}.
\]
Otherwise, $k\ge l>2^\nu$, and we argue that
\begin{align*}
  \al_2\cbots\al_{l}
  &=   (\al_2\cbots\al_{2^\nu})(\al_{2^\nu+1}\cbots\al_l) \\
  &\le (\al_{k+1}\cbots\al_{k+2^\nu-1})(\al_{2^\nu+1}\cbots\al_l) \\
  &\le (\al_{k+1}\cbots\al_{k+2^\nu-1})(\al_{k+2^\nu}\cbots\al_{k+l-1}) \\
  &= \al_{k+1}\cbots\al_{k+l-1}.
\end{align*}
This is equivalent to $m_km_l\le m_{k+l-1}$. So we indeed obtain \ref{9} with $\lm=1$.
\goodbreak

\zg9\imp4.
If $m$ is strongly submultiplicative, then
\begin{align*}
  m_rm_{k_1}\cbots m_{k_r}
  &\le \lm m_{k_1+r-1}m_{k_2} \cbots m_{k_r} \\
  &\le \lm^2 m_{k_1+k_2+r-2} \cbots m_{k_r}
  \le \cbots \\
  &\le \lm^r m_{k_1+\bots+k_r}.
\end{align*}  

\zg4\imp5. This again is obvious, as $k\ge r$.
\end{proof}

The properties of being \fdb and \asm are closely related but not equivalent. 
Examples to this effect are given in Appendix~\ref{a-xmp}.
But if $m$ is~\fdb, then the situation is clear cut.

\begin{lem}
Suppose $m$ is \fdb. Then $m$ is \asm if and only if \optx{$m$ is \qtt{analytic}:}
\[
  \al \defeq \inf_{n\ge1 \vph{1^1}} \, m_n\rn > 0.
\]
\end{lem}

\begin{proof}
Suppose $m$ is \fdb. If $\al>0$, then
\[
  m_{k_1}\cbots m_{k_r}
  \le \al^{-r} m_r m_{k_1}\cbots m_{k_r} 
  \le \al^{-r}\lm^k m_k
\]
with $k = k_1+\bots+k_r \ge r$. Hence $m$ is also \asm. Conversely, if $m$ is \asm, then $m$ is also almost increasing by Lemma~\ref{prop-inv}. The latter property includes that $\lm m_n\rn\ge m_1$ for all $n$, whence $\al>0$.
\end{proof}

The condition $\al>0$ amounts to $C^\om\subset E^m$ -- see Lemma~\ref{a-prop}. So whenever $E^m$ contains all the analytic functions, \fdb implies \asm. The converse, however, is not true without further assumptions. To this end,  the following property is usually required, where $\mp$ denotes the \qtt{left shift} of $m$ defined by $\mp_{n}=m_{n+1}$ for $n\ge1$.

\begin{lem}
The following properties are equivalent.
\enum
\item
$\mp \msub m$.
\item 
$m$ is \optx{\qtt{stable under differentiation} or}
\qtt{diff-stable}:\q
\mmx
$\dl \defeq \sup_{n\ge2} \smash{\pas{{m_{n}}/{m_{n-1}}}\rn} < \iny$.
\[
  \dl \defeq \sup_{n\ge2} \pas{\frac{m_{n}}{m_{n-1}}}\rn < \iny.
\]
\item
$\sup_{n\ge1} \mu_n\rn < \iny$.
\endenum
\end{lem}

\begin{proof}
On one hand,
\[
  \mp\msub m \eequi m_{n+1} \le \smash{\lm^n m_{n}}, \q n\ge1.
\]
On the other hand,
\[
  \frac{m_{n}}{m_{n-1}} = \frac{(n-1)\fac}{n\fac}\frac{M_{n}}{M_{n-1}} = \frac{\mu_{n}}{n}.
\]
From this the equivalence of all three statements follows.
\end{proof}

\begin{lem}  \label{prop-diff}
If $m$ is \asm and diff-stable, then $m$ is \fdb.
\end{lem}

\begin{proof}
Given $m_rm_{k_1}\cbots m_{k_r}$ we arrange the factors so that $k_1=\bots =k_s=1$ and $k_{s+1},\bots,k_r>1$ for some $0\le s\le n$.
Then $k_{s+1}+\bots+k_r=k-s$, and making $\dl\ge m_1$ we get
\begin{align*}
  m_rm_{k_1}\cbots m_{k_r}
  &\le  m_r m_1^s \, \dl m_{k_{s+1}-1}\cbots \dl m_{k_r-1} \\
  &\le \dl^r m_r m_{k_{s+1}-1}\cbots m_{k_r-1} \\
  &\le \dl^r\lm^k m_k.
\end{align*}
Hence, $m$ is also \fdb.
\end{proof}

The assumption of diff-stability, however, represents a severe growth restriction and is certainly not necessary, and it will never be required in the sequel. There is, however, an interesting \optx{and useful} one-to-one correspondence between \fdb and \asm weights, which  has not been noticed before and will be used in the proof of Theorem~\ref{comp-stable}.
 
\begin{lem}  \label{prop-shift}
The weight $m$ is \fdb if and only if its left shift $\mp$ is~\asm.
\end{lem}

\begin{proof}
First suppose $\mp$ is \asm. With $r\ge2$ and $\mp_0\defeq m_1$ we get
\[
  m_rm_{k_1}\cbots m_{k_r}
  = \mp_{r-1}\mp_{k_1-1}\cbots \mp_{k_r-1}
  \le \lm^{k-1} \mp_{k-1}
\]
with $k= \optx{r+k_1-1+\bots+k_r-1=} k_1+\bots+k_r$ as usual. As $\mp_{k-1}=m_k$, the weight $m$ is~\fdb.

Conversely, suppose $m$ is~\fdb and $m_1\ge1$ for simplicity.
Let $l_i=k_i+1$ for $1\le i\le r$. If $l_1 \ge r-1$, say, we write 
\[
  \mp_{k_1}\cbots\mp_{k_r}
  =   m_{l_1} m_{l_2}\cbots m_{l_{r}}
  \le m_{l_1} m_{l_2}\cbots m_{l_{r}}m_1^{\smash{l_1}-r+1}.
\]
We then apply the \fdb property to the last term with $m_{l_1}$ as the \qtt{leading factor} and $l_1$ trailing factors to get
\[
  \mp_{k_1}\cbots\mp_{k_r}
  \le \lm^{l-r+1} m_{l-r+1},
  \qq
  l=l_1+\bots+l_r.
\]
As $l-r=k$ and $m_{k+1}=\mp_k$ we get the \asm property for~$\mp$ in this case.

Otherwise, we have $l_1,\bots,l_r<r-1$. We then proceed by induction to get
\[
  m_{l_1} m_{l_2}\cbots m_{l_{r}}
  \le \lm^{n_s} m_{n_s-n_{s-1}} m_{l_a+n_{s-1}+1}\cbots m_{l_{r}}
\]
with $a=l_1$ and
\mms $
  n_s = l_2+\bots+l_{a+n_{s-1}}
$
for $s\ge1$ with $n_0=0$.
After finitely many steps, \optx{we get} $n_s-n_{s-1} \ge r-a-1-n_{s-1}$, or
\mms $
  n_s \ge r-1-l_1
$.
Now we apply the immediate estimate to obtain
\[
  m_{l_1}m_{l_2}\cbots m_{l_{r}}
  \le \lm^{n_s+n}m_n
\]
with
\begin{align*}
  n 
  &= l_{a+n_{s-1}+1}+\bots+l_{r} + l_1 + n_s - r +1  \\
  &= l_1+\bots+l_{r}-r+1 
  \\&
   = k+1
\end{align*}
as before. So also in this case we obtain the \asm property for~$\mp$.
\end{proof}

We note that for Gevrey weights we have
\[
  M_n = n\fac^{s},
  \qq
  m_n = n\fac^{s-1}.
\] 
Hence, for $s>0$ they are weakly log-convex. For $s\ge1$ they are log-convex and thus \asm and stricly \fdb. They are also diff-stable, as 
\[
  \lim_{n\to\iny} \pas{\frac{m_{n}}{m_{n-1}}}\rn
  = \lim_{n\to\iny} n^{s/n}
  = 1.
\]

\section{Composition}  \label{comps}

To study the composition of \m{E}-maps we first consider formal power series, which avoids the cumbersome Faà di Bruno formula.
We employ the standard notation 
\[
  \sum_{k\in\Lm} f_kx^k \cle \sum_{k\in\Lm} g_kx^k
\]
for two formal power series in \m{\vq}-space, when 
\[
  \ab{f_k} \vle g_k
  \eeequi
  \n*{f_{k,i}}\le g_{k,i}, \qq 1\le i\le\vq,
\]
holds for all coefficients.

To simplify notation we consider a \m{s}-dimensional weight as a weight on any lower-dimensional index space as well by identifying $(k_1,\bots,k_{\vq})$ with $(k_1,\bots,k_{\vq},0,\bots,0)$.
In other words, we add dummy coordinates to make the dimensions equal.

\begin{ml} \label{comp-wm}
Let $g$ and $h$ be two formal power series without constant terms. If\/ $m$ and $w$ are two weights such that, with some $\lm>0$,
\[
  w_lm_{k_1}\cbots m_{k_r} \le \lm^k m_k 
\]
for all $l\ne0$ and $k_1,\bots,k_r\ne0$ such that
$r=\n{l}$ and $k_1+\bots+k_r = k$,
then
\[
  \Md^m_0(g\cmp h) \cle \Md^w_0\!g\cmp\Md^m_0 h \cmp \lm.
\]
\end{ml}

\begin{proof}
Write
\[
  g = \sum_{l\ne0} g_lz^l,
  \qq
  h = \sum_{k\ne0} h_kx^k,
\]
where $l=(l_1,\bots,l_{\vq})$ and $k=(k_1,\bots,k_s)$. Then
\[
  g\cmp h
  = \sum_{l\ne0} g_l {\pas3{ \,\sum_{k\ne0} h_kx^{k}}^{\lox{l}}}
  = \sum_{k\ne0} \pas3{ \,\sum_{l\ne0}\,\sum_{k_1,\bots,k_r} g_l\hd_{k_1}\cbots \hd_{k_r}\! }x^k,
\]
where the last sum is taken over all $l\ne0$ and $k_1,\bots,k_r\ne0$ such that $r=\n{l}$ and $k_1+\bots+k_r = k$,
and where $\hd_k$ stands for a certain component of the \m{\vq}-vector $h_k$ which we do not need to make explicit. 

By hypotheses,
$
  m_k \ge \lm^{-k}w_lm_{k_1}\cbots m_{k_r}
$
in all cases.
Therefore,
\begin{align*}
  \Md^m_0(g\cmp h)
  &\cle \sum_{k\ne0} {\pas3{ \,\sum_{l\ne0}\,\sum_{k_1,\bots,k_r}
         \ab{g_l}\n*{\hd_{k_1}\!}\cbots\n*{\hd_{k_r}\!} }} \frac{x^k}{m_k}
         \\
  &\cle \sum_{k\ne0} {\pas3{ \,\sum_{l\ne0}\,\sum_{k_1,\bots,k_r}
         \frac{\ab{g_l}}{w_l}
         \frac{\n*{\hd_{k_1}\!}}{m_{k_1}}\cbots\frac{\n*{\hd_{k_r}\!}}{m_{k_r}} }} (\lm x)^k.
\end{align*}
This is tantamount to
\begin{align*}
  \Md^m_0(g\cmp h)
  \cle \sum_{l\ne0} 
         \frac{\ab{g_l}}{w_l} {\pas3{ \,\sum_{k\ne0} \frac{\ab{h_k}}{m_k}(\lm x)^k }^{\lox{l}}}
  = \Md^w_0g\cmp \Md^m_0h \cmp \lm.
  \qedhere
\end{align*}
\end{proof}

\begin{thm} \label{comp-gen}
Let $g\in E^w$ and $h\in E^m$ and suppose $g\cmp h$ is well defined. If\/ $w$ and $m$ satisfy the assumption of the preceding lemma, then $g\cmp h\in E^m$. In particular,
\[
  \nn{g\cmp h}^m_{a,r} 
  \le \nn{g}^w_{h(a),\rh},
  \qq
  \rh = \nn{h}^m_{a,\lm r}.
\]
\end{thm}

\begin{proof}  \def\so{\rh_0}
Let $h$ be of class $E^m$ near $a$ and $g$ be of class $E^w$ near $b=h(a)$. 
Then $g\cmp h$ is well defined near~$a$.
Moreover, there is a neighborhood $V$ of $b$ and a $\so$ such that
$
  \nn{g}^w_{V\!,\so} < \iny
$.
There is another neighborhood $U$ of $a$ and an~$r$ such that $h(U)\subset V$ and
$
  \nn{h}^m_{U,r} < \iny
$.
As this semi-norm does not include a constant term, we can make it as small as we like by making $r$ small. So in particular we can choose $r$ so that
\[
  \nn{h}^m_{U,\lm r} < \so.
\]
Applying the preceding lemma to the formal Taylor series expansions of $g$ at $b$ and $h$ at~$a$ we obtain
\[
  \Md^m_a(g\cmp h) \cle \Md^w_b\!g \cmp \Md^m_ah \cmp \lm.
\]
Considering each component of $g\cmp h$ separately we conclude that
\begin{align*}
  \nn{g\cmp h}^m_{a,r} 
  &=   \Md^m_a(g\cmp h)(\rbr) \\
  &\le \Md^w_b\!g \cmp \Md^m_ah(\lm r,\bots,\lm r) \\
  &\le \Md^w_b\!g(\nn*{h_1}^m_{a,\lm r},\bots,\nn*{h_{\vq}}^m_{a,\lm r}). 
\end{align*}
With $\max_{1\le i\le\vq} \nn{h_i}^m_{a,\lm r} = \nn{h}^m_{a,\lm r} \eqdef \rh < \so$ we get
\[
  \nn{g\cmp h}^m_{a,r} 
  \le \Md^w_b\!g(\rh,\bots,\rh)
  = \nn{g}^w_{b,\rh}
  < \iny.
\]
As this holds locally around any point in the domain of~$h$ and hence of $g\cmp h$, the latter is also of class~$E^m$.
\end{proof}

\section{Characterizations of $E^m$}  \label{s-char}

We now characterize those \m{E^m}-spaces, which are stable under composition. There are two distinct cases, the holomorphically or \m{C^\om}-closed and the \m{E^m}-closed spaces. 
Neither of these characterizations require the property of being closed under derivation. 
These results are obviously optimal and more general than those in \cite{RaS-14,RaS-15}. 

The first theorem already appeared in \cite{Sid-90} but apparently did not receive much attention. Its roots go back to \cite{Rud-62,BoH-62}.

\begin{thm} \label{hol-stable}
The following statements are equivalent.
\enum[z]
\item•1
$m$ is \asm.
\item•2
$E^m$ is \qtt{holomorphically closed}:\en for $g\in C^\om$ and $h\in E^m$, also $g\cmp h\in E^m$.
\item•3
$E^m$ is \qtt{inverse closed}:\en for $h\in E^m$, also $1/h\in E^m$ wherever defined.
\endenum
\end{thm}

\begin{proof}
\zg1\imp2.
Recall that $C^\om=E^w$ with $w=(1)_{n\ge1}$. If $m$ is \asm, then $w$ and $m$ satisfy the hypothesis of the Main Lemma~\ref{comp-wm}. Thus, if $g\in C^\om$ and $h\in E^m$, then also $g\cmp h\in E^m$ by Lemma~\ref{comp-gen}.

\zg2\imp3.
This is obvious, as $z\mapsto 1/z$ is holomorphic for $z\ne0$.

\zg3\imp1.
Assume for simplicity that $E^m$ consists of functions on an interval around~$0$. Let $\eta$ be the characteristic function of Lemma~\ref{char-func} and $\rh\maps z\mapsto (1-z)\inv$. The Taylor coefficients of $\rh$ are $\rh_r=1$ for all $r\ge1$, so
\begin{align*}
  \Td_0(\rh\cmp\eta)
  &= \sum_{n>0} \pas3{ \,\sum_{r>0}\,\sum_{n_1+\bots+n_r=n\vph0} 
       \rh_r \eta_{n_1}\cbots \eta_{n_r} } x^n \\
  &= \sum_{n>0} \j^{n-r} \pas3{ \,\sum_{r>0}\,\sum_{n_1+\bots+n_r=n\vph0} s_{n_1}\cbots s_{n_r} } x^n
\end{align*}
and
\[
  \nn{\rh\cmp\eta}^m_{0,r}
  \ge \sum_{n>0}\,\sum_{r>0}\,\sum_{n_1+\bots+n_r=n\vph0} \frac{m_{n_1}\cbots m_{n_r}}{m_n} \, r^n.
\]
Hence, for $\rh\cmp\eta$ to be in $E^m$, the weight $m$ has to be \asm.
\end{proof}


\begin{thm} \label{comp-stable}
The following statements are equivalent.
\enum[z]
\item•1
$m$ is \fdb.
\item•2
$E^m$ is \qtt{composition closed}:\en with $g,h\in E^m$, also $g\cmp h\in E^m$.
\item•3
$E^m$ is \qtt{inversion closed}:\en
if\/ $g$ is a local diffeomorphism in $E^m$, then its local inverse $g\inv$ is also in~$E^m$.
\endenum
\end{thm}

The first results of this kind are apparently due to Gevrey~\cite{Gev-18} and Car\-tan~\cite{Car-40}.
For log-convex weights this was shown by Roumieu~\cite{Rou-62}, and for more details see \cite{FrG-06,Ide-89,RaS-14,RaS-15} among many others.
The necessity of the \fdb condition generalizes results in \cite{RaS-15} and is new, as is the proof of \ref{1}$\imp$\ref{3}.

Note also that when $E^m$ is inversion closed, the corresponding implicit function theorem holds as well.

\begin{proof}
\zg1\imp2.
If\/ $m$ is \fdb, then we can apply Lemma~\ref{comp-gen} with $w=m$ and conclude that $E^m$ is stable under composition.

\zg1\imp3.   
We may reduce the problem to a local diffeomorphism $g=\id+\gh$ in the neighbourhood of its fixed point~$0$, where the hat denotes higher order terms. Its local inverse may be written as $\rh=\id-\rhh$. From $g\cmp \rh=\id$ we obtain
$\rhh = \gh\cmp\rh$ 
and hence
\mms $
  D\rhh = (D\gh\cmp\rh)(I+D\rhh)
$.
With $g\in E^m$ \opt{we have} $D\gh\in E^{\mp}$. 
As $\mp$ is \asm by Lemma~\ref{prop-shift}, $E^{\mp}$ is holo\-morphically closed, so in particular an algebra, and we may normalize $\mp$ so that $\nn{uv}^{\mp}_{a,r}\le\nn{u}^{\mp}_{a,r}\nn{v}^{\mp}_{a,r}$. Furthermore, there exists a neighbourhood $U$ of $0$ and an $r>0$ such that
\[
  \nn{D\gh}^{\mp}_{U,r} \le 1/2,
\]
and another neighboorhoud $V$ so that $\rh(V)\subset U$.

As we do not know yet whether $\nn*{D\rhh}^{\mp}_{V,r}$ is finite, we first consider its \m{N}-th Taylor polynomial $T^N\!D\rhh$. As in the above equation its terms do not depend on higher oder terms we conclude that
\[
  \nn*{T^N\!D\rhh}^{\mp}_{V,r}
  \le \nn{D\gh}^{\mp}_{U,r}(1+\nn*{T^N\!D\rhh}^{\mp}_{V,r}).
\]
Hence,
\[
  \nn*{T^N\!D\rhh}^{\mp}_{V,r} \le 2 \nn{D\gh}^{\mp}_{U,r}.
\]
As this hold for all $N$, this implies that $D\rhh\in E^{\mp}$ and consequently that $\rh\in E^m$.

\zg2\imp1.
Assume again that $0$ is in the domain under consideration. For the characteristic \m{E^m}-function $\eta$ of Lemma~\ref{char-func} we obtain
\begin{align*}
  \Td_0(\eta\cmp\eta)
  &= \sum_{n>0} \pas3{ \,\sum_{r>0}\,\sum_{n_1+\bots+n_r=n\vph0} \eta_r\eta_{n_1}\cbots\eta_{n_r} } x^n \\
  &= \sum_{n>0} \j^{n-1} \pas3{ \,\sum_{r>0}\,\sum_{n_1+\bots+n_r=n\vph0} s_rs_{n_1}\cbots s_{n_r} } x^n,
\end{align*}
as the ‘signs’ of all terms for $x^n$ combine to 
\mms
$\j^{r-1}\j^{n_1-1}\cbots\j^{n_r-1} = \j^{r-1}\j^{n-r} = \j^{n-1}$.
It follows that
\[
  \nn{\eta\cmp\eta}^m_{0,r} 
  \ge \sum_{n>0} \,\sum_{r>0}\,\sum_{n_1+\bots+n_r=n\vph0} \frac{m_rm_{n_1}\cbots m_{n_r}}{m_n} \, r^n.
\]
Hence, for $\eta\cmp\eta$ to be in $E^m$, the weight $m$ has to be \fdb.

\zg3\imp1.  \let\rhr=r
We may modifiy the linear term of the characteristic function $\eta$ so that $\eta=\id+\eth$. This is a local diffeomorphism at~$0$ with an inverse function of the form $\rh=\id-\rhh$ for which we make the ansatz
\[
  \rhh = \sum_{n\ge2} \rh_nx^n = \sum_{n\ge2} \j^{n-1}\rhr_nx^n.
\]
Then $\eta\cmp\rh=\id$ is equivalent to $\rhh=\eth\cmp\rh$ or
\[
  \sum_{n\ge2} \rh_nx^n
  = \sum_{n\ge2} \pas3{ \,\sum_{r>0}\,\sum_{n_1+\bots+n_r=n\vph0} \eta_r\rh_{n_1}\cbots\rh_{n_r} } x^n.
\]
Comparison of coefficients leads to
\[
  \rhr_n = s_n + \sum_{1<r<n} \,\sum_{n_1+\bots+n_r=n\vph0} s_r\rhr_{n_1}\cbots \rhr_{n_r},
  \qq
  n\ge2,
\]
where $\rhr_1=1$. It follows that $\rhr_n\ge s_n\ge m_n$ for all $n\ge1$, and finally that
\[
  \rhr_n \ge m_r m_{n_1}\cbots m_{n_r}.
\]
For $\rh$ to be in $E^m$, the weight $m$ thus has to be \fdb.
\end{proof}

\opt{For the sake of completeness and comparison we mention}

\begin{thm}
The following statements are equivalent.
\enum[z]
\item•1
$m$ is diff-stable.
\item•2
$E^m$ is \qtt{derivation closed}:\en with $f\in E^m$ also $f' \in E^m$.
\endenum
\end{thm}

\begin{proof}
It follows from the definitions that with $f\in E^m$ we have $f'\in E^{\mp}$. If $m$ is diff-stable, then $\mp\msub m$ and $E^{\mp}\subset E^m$ and thus $f'\in E^m$.
The converse is proven with characteristic functions as usual.
\end{proof}

\begin{cor}
If\/ $E^m$ is holomorphically and derivation closed, then $E^m$ is also composition closed.
\end{cor}

\begin{proof}
If $E^m$ is holomorphically closed, than $m$ is almost submultiplicative by Theorem~\ref{hol-stable}. If $E^m$ is also derivation closed, then $m$ is also diff-stable and thus \fdb by Lemma~\ref{prop-diff}. So $E^m$ is composition closed by Theorem~\ref{comp-stable}.
\end{proof}

\paragraph{Generalizations}
The preceding results hold for \emph{local} Denjoy-Carleman classes~$E^m$, where $f$ belongs to~$E^m$, if any point in its domain has a \emph{neighbourhood} $U$ such that
\[
  \sup_{n\ge1} \frac{1}{n\fac}\frac{\nn*{D^n\f}_U}{m_n}\,r^n < \iny
\]
for some $r>0$.
But as noted in \cite{Sid-90} they also hold for Denjoy-Carleman classes $E^m(I)$ for an \emph{arbitrary} interval~$I$, where $f$ belongs to $E^m(I)$, if 
\[
  \sup_{n\ge1}  \frac{1}{n\fac}\frac{\nn*{D^n\f}_I}{m_n}\,r^n < \iny
\]
for some $r>0$, with $\nn\cd_I$ denoting the usual sup-norm on~$I$.
If $I$ is compact, these classes coincide with the local classes, and there is nothing new. Otherwise, it is no longer true that 
\[
  E^m(I) = E^{\mb}(I)
\]
with a \emph{weakly log-convex} sequence $\mb$. Instead, the regularized weight $\mb$ has to be defined slightly differently, depending on the nature of the interval~$I$ -- see ~\cite{Sid-90} and \cite{Man-52}. But the crucial fact is that also in this case characteristic \m{E^m}-functions exist, and the proofs remain essentially the same.

\section{Flows}  \label{flows}

\def\mx{\dot m}
\def\mtm{\mt\cd m}

We now establish the \m{E^m}-regularity for the flows of ode's. To this end we need to assume the weights to be strictly~\fdb. By Lemma~\ref{prop-fdb} this holds for log-convex weights, which is a standard assumption in this context\optx{ anyway}.

For the classical existence theorem for ode's it makes sense to allow for different differentiability properties with respect to time and space.
So let $\mx=\mtm$ be the product of two weights $\mt$ and $m$ defined \optx{in the obvious way} by 
\[
  \mx_{\dot n} = (\mtm)_{(\tilde n,n)} = \mt_{\tilde n}\cd m_n.
\]
Then $E^{\mx}=E^{\mtm}$ is a space of functions which are of class $E^{\mt}$ and $E^{m}$ with respect to two different subsets of coordinates, in this case time and space. 

The following theorem generalizes and extends results in \cite{Kom-80,Yam-91} in the finite dimensional case. 
It extends in the obvious way to vector fields depending on parameters.

\begin{thm}\label{ode}
Let $v$ be a time-dependent vector field of class~$E^{\mtm}$ with respect to time and space. If\/ $\mtm$ is strictly \fdb, then its flow~$\ph$, wherever defined, is also of class $E^{\mtm}$.\end{thm}

\begin{proof}
Consider the Taylor expansion of the flow $\ph$ at a point $a=(t,x)$ with $t=0$. The differential equation $\del_t\ph = v\cmp \ph$ and $\ph(0)=a$ lead to
\[
  \del_t(T_a\ph) 
  = T_a(\del_t\ph)
  = T_a(v\cmp\ph)
  = T_av \cmp \Td_a\ph.
\]
On the left hand side the constant term drops out so we get
\[
  \del_t(\Td_a\ph) = T_av \cmp \Td_a\ph.
\]
Applying the Main Lemma and writing $M$ for $M^{\mt\cd m}$,
\[
  \del_t(\Md_a\ph) \cle M_av \cmp \Md_a\ph.
\]
Now we may define recursively a formal Taylor series $g$ through
\[
  \del_t g = M_av\cmp g,
  \qq
  \eval2{g}_{t=0} = \id.
\]
We then have
\mms $
  M_a\ph \cle T_ag
$.
But $g$ is the formal solution of the differential equation $\del_t g = M_av\cmp g$. As $M_av$ is analytic on some neighborhood of~$a$, $g$ is analytic on some smaller neighborhood of~$a$. Thus, there exists some positive $r$ such that
\[
  \nn*{\ph}^{\mtm}_{a,r} 
  = \nny*{\Md_a\ph(r)}
  \le \nny{g(r)}
  < \iny.
\]

The seminorm $\nn{v}^{\mtm}_{a,r}$ is uniformly bounded on some neighborhood $V$ of~$a$. By contuinity we can choose another neighborhood $U\subset V$ such that $\ph(U)\subset V$. It then follows that also
\mms $
  \nn*{\ph}^{\mtm}_{U,r} < \iny
$.
So this local map is of class~$E^{\mtm}$.

Finally, at any point the flow map $\ph$ can be written as the composition of finitely many local flow maps. As $E^{\mtm}$ is stable under composition, the general statement follows.
\end{proof}

For the sake of completeness we consider the corresponding Cauchy-Kowaleskaya problem for pde's in Appendix~\ref{s-CK}.

\section{Proof of the Siegel-Sternberg theorem -- Part I}
\label{part-1}

We now turn to the proof of the Siegel-Sternberg theorem. It consists of two very different parts. First, a small divisor problem is solved to linearize the map $g$ up to a flat term -- without using explicit small divisor conditions. Second, hyperbolicity is used to also remove that flat term within the class of \m{E^m}-maps.

\begin{thm}
Consider a smooth map $g$ of class $E^m$ in a neighbourhood of a fixed point. If its linearization at the fixed point is nonresonant, then $g$ can be linearized up to a flat term by a local diffeomorphism $\ph$ of class $E^{m\star w}$, where $w$ is any weight dominating the associated nonresonance function $\Om$ such that $m\star w$ is log-convex and strongly non-analytic.
\end{thm}

Here, \emph{strongly non-analytic} means that
\[
  \sup_{q\ge1} \frac{\mu_q}{q} \sum_{k\ge q} \frac{1}{\mu_k} < \iny.
\]
For instance, Gevrey weights with $M_k = k\fac^s$ and $\mu_k = k^s$ have this property for all~$s>1$.
This is the necessary and sufficient condition for the \emph{Borel map}
\[
  T_a\maps E^m \to F^m, \q f \mapsto T_af = \sum_{k\in\Lm} \frac{D^k\f(a)}{k\fac}(x-a)^k
\]
to be \emph{onto}, where $F^m$ denotes the space of formal power series at a fixed point~$a$ supplied with the norms $\nn\cd^m_{a,r}$.
See~\cite{Pet-88} for this result and more details. In other words, if -- and only if -- the weight $m$ satisfies the condition of strong non-analyticity, then \emph{any} formal power series 
$\sum_{k\in\Lm} f_k(x-a)^k$ with
\[
  \sum_{0\ne k\in\Lm} \frac{\n{f_k}}{m_k}r^k < \iny
\]
for some $r>0$ is the Taylor series of an \m{E^m}-function defined in some neighbourhood of~$a$.

\begin{proof}
The construction of a formal solution to the equation
\mms $
  g\cmp\ph = \ph\cmp\Lm
$
is simple. Writing $\ph=\id+\phh$ and $g=\Lm+\gh$, expanding into formal power series
\[
  \ph = \sum_{\n{k}\ge1} \ph_kx^k, 
  \qq
  \gh = \sum_{\n{l}\ge2} g_lx^l,
\]
and collecting terms linear in $\ph_k$ on the left side,
this equation is equivalent to
\[
  \sum_{\n{k}\ge2} \Lm_k\ph_kx^k = \sum_{\n{l}\ge2} g_l \pas3{ \sum_{\n{k}\ge1} \ph_kx^k}^{\lox{l}}
\]
with
$
  \Lm_k = \lm^k I - \Lm
$.
As all $\Lm_k$ are regular by nonresonance, all coefficients $\ph_k$ are determined recursively in terms of the $g_l$. 

It remains to obtain explicit bounds on their growth. As $\Lm$ is semi-simple by nonresonance we can choose an appropriate norm so that
\[
  E_k \defeq \nn*{\Lm_k\inv} = \max_{1\le i\le s} \n*{\lm^k-\lm_i}\inv.
\]
Hence
\[
  \n{\ph_k} \le E_k \sum_{l\le k} \,
                    \sum_{k_1+\bots+k_r=k} \n{g_l} \n*{\ph_{k_1}}\cbots\n*{\ph_{k_r}}
\]
with $r=\n{l}\ge2$ as usual.
For a strict \fdb weight $m$, we thus have
\[
  \frac{\n{\ph_k}}{m_k}
  \le E_k \sum_{l\le k} \, \sum_{k_1+\bots+k_r=k} 
             \frac{\n{g_l}}{m_l} \frac{\n*{\ph_{k_1}}}{m_{k_1}}\cbots\frac{\n*{\ph_{k_r}}}{m_{k_r}}.
\]
Passing to an equivalent weight~$m$, we may assume that
$
  \nn{\gh}^m_{0,1} \le 1
$,
or
\[
  \frac{\n{g_l}}{m_l} \le 1, \qq \n{l}\ge2.
\]
Following Siegel \cite{Sie-42} we then obtain
\[
  •ph-m
  \frac{\n{\ph_k}}{m_k} \le \sg_{\n{k}}\Dl_k
\]
with the inductively defined sequences
\[ \begin{aligned}
  \sg_1 &= 1, \qq&
  \sg_n &= \sum_{n_1+\bots+n_r=n} \sg_{n_1}\cbots\sg_{n_r}, &
  n&\ge2, \\
  \Dl_e &= 1, &
  \Dl_k &= E_k \max_{k_1+\bots+k_r=k} \Dl_{k_1}\cbots\Dl_{k_r}, &\qq
  \n{k}&\ge2,
\end{aligned} \]
where $e$ denotes any multiindex of length~$1$ and $r\ge2$.

The $\sg_n$ are the coefficients of the unique \emph{analytic} function $\sg=\id+\hat\sg$ solving $(\sg-t)=\sg^2/(1-\sg)$. 
Hence, they grow like a power of $n$, and 
\[
  \sup_{n\ne0} \frac{1}{n} \log \sg_n < \iny.
\]
The $\Dl_k$, on the other hand, represent the accumulation of near resonances and usually grow more rapidly.
To obtain useful bounds we essentially follow Bruno's argument \cite{Bru-72,Pos-86}.

In the definition of $\Dl_k$ the maximum is attained for some decomposition $k=k_1+\bots+k_r$, which we may choose in some definite way. Proceeding like this with every factor $\Dl_{k_i}$ we end up with some well defined decomposition
\[
  \Dl_k = E_{l_0} E_{l_1}\cbots E_{l_s},
  \qq
  l_0=k, \q 2\le \n{l_1},\bots,\n{l_s}<\n{k}.
\]
Let
\[
  N_n(k)
  = \card \set1{E_{l_i} \mathop\mid \Dl_k : E_{l_i}>\eta\Om(n)}, 
  \qq
  \eta = 4\max_{1\le i,j\le s} \n*{\lm_i/\lm_j},
\]
be the number of factors in $\Dl_k$ exceeding $\eta\Om(n)$.
The following is the key estimate which is due to Bruno \cite{Bru-72} for admissible small divisiors. But in fact it is completely independent of any growth properties of the function~$\Om$. 

\begin{cl}
For $n\ge2$,
\[
  N_n(k) \le \ccases{0, &\n{k}\le n, \\[3pt] 2\Frac{\n{k}}{n}-1, & \n{k}>n.}
\]
\end{cl}

With this bound the estimate of $\Dl_k$ is straightforward. We have
\[
  \card \set1{ E_{l_i}\mathop\mid \Dl_k : 
          \eta\Om(2^\nu) \le E_{l_i} < \eta\Om(2^{\nu+1}) } \le \frac{\n{k}}{2^{\nu-1}},
\]
and we only need to consider those $\nu$ with $2\le 2^\nu<\n{k}$. Therefore,
\begin{multline*}
  \frac{1}{\n{k}} \log \Dl_k
  \le \smash[b]{\sum_{1\le\nu<\log_2\n{k}}} \frac{1}{2^{\nu-1}} \log \eta\Om(2^{\nu+1}) \\[-5pt]
  \le 2\log\eta + 2 \sum_{2\le2^\nu<\n{k}} \frac{ \log \Om(2^{\nu+1})}{2^{\nu}}.
\end{multline*}
So for any weight $w$ dominating $\Om$ in the sense that
\[
  \sum_{\nu\le\log_2\n{k}} \frac{\log \Om(2^{\nu+1})}{2^\nu} \le a + \frac{\log w_k}{\n{k}},
  \qq
  \n{k}\ge2,
\]
we have
\[
  \sup_{k\ne0} \frac{1}{\n{k}} \log \frac{\Dl_k}{w_k} < \iny.
\]
Together with~\eqref{ph-m} we conclude that
\[
  \sup_{k\ne0} \frac{1}{\n{k}} \log \frac{\n{\ph_k}}{m_kw_k} < \iny.
\]

So, at least formally, $\ph$ is of class $E^{m\star w}$. Now, if $m\star w$ is analytic, then $\ph$ is also convergent, and we are finished. Otherwise, we may increase $w$ if necessary so that $m\star w$ is log-convex and strongly non-analytic. Then, by the surjectivity of the Borel map~$T_0$, there exists a \emph{bona fide} map
\[
  \ph=\id+\phh\in E^{m\star w},
  \qq
  T_0\phh = \sum_{\n{k}\ge2} \ph_kx^k.
\]
Its local inverse $\ph\inv$ is also of class $E^{m\star w}$ by Theorem~\ref{comp-stable}, so it is a local \m{E^m}-diffeomorphism. Consequently,
\[
  \ph\inv\cmp g\cmp \ph = \Lm + h
\]
is also a $E^{m\star w}$-map with a remainder, which is flat at the origin. This proves the theorem.
\end{proof}

\begin{proof}[Proof of the Counting Lemma]
Fix $n$ and assume $\n{k}>n$.
By construction,
\[
  \Dl_k = E_k\Dl_{k_1}\cbots\Dl_{k_r}
\]
with $k_1+\bots+k_r=k$ and $\n{k}>\n{k_1}\ge\bots\ge\n{k_r}\ge1$. In this decomposition, only $\n{k_1}$ may be greater than
\[
  K \defeq \max(\n{k}-n,n).
\]
If this is the case, we decompose $\Dl_{k_1}$ in the same way. Repeating this step at most $n-1$ times, we finally obtain a decomposition
\[
  \Dl_k = E_k E_{k_1} \cbots E_{k_\nu} \Dl_{l_1}\cbots\Dl_{l_\mu},
  \qq
  \nu\ge0, \q \mu\ge2,
\]
where 
\[
  k>k_1>\bots>k_\nu,
  \q
  l_1+\bots+l_\mu=k,
  \q
  \n*{k_\nu} > K \ge \n{l_1}\ge\bots\ge\n*{l_\mu}.
\]
The point is that \emph{at most one} of the \m{E}'s count towards $N_n(k)$. This is the consequence of Siegel's observation~\cite{Sie-42} that if 
\[
  E_k, E_{k'} > \eta\Om(n), \qq k>k',
\]
then an elementary calculation shows that
\[
  \n*{\lm^{k-k'}\lm_i-\lm_j} > \Om(n)
\]
for some $1\le i,j\le s$. By the definition of $\Om$, this implies that $\n{k-k'}+1>n$, or
$
  \n{k-k'} \ge n
$.
So we obtain
\[
  N_n(k) \le 1+N_n(l_1)+\bots+N_n(l_\mu).
\]
Choose $0\le\rh\le\mu$ so that 
$\n*{l_\rh}>n\ge\n*{l_{\rh+1}}$. Arguing by induction we obtain
\begin{align*}
  N_n(k)
  &\le 1+N_n(l_1)+\bots+N_n(l_\rh) 
   \\[-8pt]
  &\le 1+2\frac{\n{l_1+\bots+l_\rh}}{n}-\rh 
   \en\le\en \ccases{ 1, &\rh=0, \\[4pt]
                2 \Frac{\n*{k}-n}{n}, & \rh=1, \\[4pt]
                2 \Frac{\n*{l_1+\bots+l_\rh}}{n}-1, & \rh\ge2,
              } 
   \\[-8pt]
  &\le 2\frac{\n{k}}{n}-1.
\end{align*}
This proves the Counting Lemma.
\end{proof}

\section{Proof of the Siegel-Sternberg theorem -- Part II}

\begin{thm} \label{th-flat}
Suppose $g$ has a hyperbolic fixed point with a flat nonlinearity of class~$E^m$. If\/ $m$ is log-convex and strongly non-analytic, then the nonlinearity can be removed by a local diffeomorphism of class~$E^{\mp}$.
\end{thm}

\begin{proof}
Place the fixed point at the origin.
As its linearization $\Lm=Dg(0)$ is supposed to be hyperbolic, there exists a splitting $E^s\x E^u$ of the entire space into an attracting and expanding invariant subspace of~$\Lm$. 
The corresponding local stable and unstable manifolds $W^s$ and $W^u$ are then given locally as graphs of smooth maps
\[
  u\maps E^s\to E^u, \qq v\maps E^u \to E^s,
\]
which vanish up to first order at the origin. In coordinates $(x,y)\in E^s\x E^u$,
\[
  \ps\maps (x,y) \mapsto (x+v(y),y+u(x))
\]
thus defines a smooth local diffeomorphism which flattens those local manifolds so that
\[
  W^s \subset E^s\x\set{0}, \qq
  W^u \subset \set{0}\x E^u.
\]
Following Irwin's construction \cite{Irw-80} of those invariant manifolds, $u$ and $v$ are as smooth as the mapping~$g$ itself. Hence they are of class~$E^m$, and so is the diffeomorphism~$\ps$. As we assume $m$ to be log-convex, its inverse map $\ps\inv$ is also~$E^m$, and the same holds for the transformed map
$\ps\inv\cmp g\cmp \ps$.

So we may consider a map $g=\Lm+h$ of class~$E^m$ in coordinates such that $\Lm$ is hyperbolic, its local invariant manifolds are straightened out, and its nonlinearity $h$ is still flat at the origin.

The next step is to split $h$ into two terms which are flat at either $W^s$ or~$W^u$.
As $m$ is supposed to be strongly non-analytic there exists \cite{Pet-88} a family $(\ch_k)_{k\in\Lm^s}$ of \m{E^m}-functions on $E^s$ such that 
\[
  D^{k'}\!\ch_k(0) = \dl_{kk'}
\]
for all derivatives, and for any \m{E^m}-function $f$ on~$E^s$,
\[
  \phi(x) = \sum_{k\in\Lm^s} D^k\f(0)\ch_k(x)
\]
defines again an \m{E^m}-function. In other words, $\phi$ is an \m{E^m}-extension of the \c^\iny-jet of $f$ at the origin.

Fix such a family $(\ch_k)$ and define $h^u$ by
\[
  h^u(x,y) = \sum_{k\in\Lm^s} D^k h(0,y)\ch_k(x).
\]
Then $h^u$ is of class $E^m$ in all coordinates, $\eval{h^u}_{W^u} = h$, and 
\[
  \eval{D^kh^u}_{W^u} = \eval{D^kh}_{W^u}, \qq \eval{D^lh^u}_{W^s} = 0
\]
for all $k\in\Lm^s$ and $l\in\Lm^u$, where $D^k$ and $D^l$ refer to derivatives in the direction of $E^s$ and $E^u$, respectively. Hence $h^u$ coincides with $h$ on $W^u$ and is flat on~$W^s$.

For $h^s=h-h^u$ we then immediately obtain $\eval{h^s}_{W^s} = h$ and
\[
  \eval{D^kh^s}_{W^u} = 0
\]
for all~$k\in\Lm^s$.
So $h^s$ coincides with $h$ on $W^s$ and is flat on~$W^u$. Hence,
\[
  h = h^s+h^u
\]
is a decomposition of $h$ with the required properties.

We remark that the definition of $h^u$ amounts to a special version of the Whitney extension theorem for Denjoy-Carleman classes, but avoids the many restrictions on the weight $m$ which seem to be necessary in the general case \cite{Bru-80,Val-11}.

We are now ready to remove the nonlinearity $h$ by a deformation argument first applied in~\cite{Mos-65}.
In general, let $(g_t)_{0\le t\le 1}$ be a family of maps, and suppose we are looking for a family of diffeomorphisms $(\ph_t)_{0\le t\le 1}$ such that
\[
  \ph_t\inv \cmp g_t \cmp \ph_t = g_0, \qq 0\le t\le 1.
\]
If we write $\ph_t$ as the time-\m{t}-map of a time dependent vector field $X_t$ so that
\mms $
  \dot \ph_t = X_t\cmp\ph_t
$,
then differentiation of $g_t\cmp\ph_t = \ph_t\cmp g_0$ leads to
\[
  \dot g_t\cmp\ph_t+Dg_t\cmp\ph_t\cd X_t\cmp\ph_t
  = X_t\cmp\ph_t\cmp g_0
  = X_t\cmp g_t\cmp\ph_t,
\]
or
\[
  •X
  X_t\cmp g_t - Dg_t\cd X_t = \dot g_t.
\]
Conversely, if $X_t$ solves the latter equation, then its flow $\ph_t$ solves the original conjugacy problem.

We apply this scheme to the family $g_t = \Lm + th$. So
\[
  \dot g_t = h, \qq 0\le t\le 1.
\]
A \emph{formal solution} is then given by
\[
  X_t = -\sum_{i\ge1} Dg_t^{-i}\cd h \cmp g_t^{i-1}  
  \qq\text{or}\qq
  X_t = \sum_{i\le0} Dg_t^{-i}\cd h \cmp g_t^{i-1},
\]
as one immediateley verifies by direct calculation.
To obtain a \emph{convergent solution} we use the hyperbolic structure of~$g$ and the corresponding splitting of~$h$. Then
\[
  •X-def
  X_t = \sum_{i\ge1} Dg_t^{-i}\!\cd h^s\cmp g_t^{i-1} - \sum_{i\le0} Dg_t^{-i}\!\cd h^u \cmp g_t^{i-1}
\]
does the job.

More precisely, using an \m{E^m}-bump function and dropping the $t$ from the notation we replace $g$ by a global diffeomorphism~$\gt$ such that $\gt=g$ on some neighbourhood of the origin and $\gt=\id$ outside a somewhat larger compact cube~$K$. This cube can be chosen so that each \m{g}-orbit starting inside $K$ converges exponentially fast to $E^u\cap K$ and $E^s\cap K$ in forward and backward time, respectively. 

\def\nno#1{\nn{#1}^{\text{o}}}

Indeed, for the stable component $\Lm_s$ of $Dg(0)$ we can choose an operator norm $\nno\cd$ so that
\[
  \nno{\Lm_s} < 1,
  \qq
  \nn{\Lm_s}^m_{a,r} = \nno{\Lm_s} r.
\]
Then we can fix $K$ and $r>0$ so that
\mms $
  \nn{g}^m_{K,r} \le \th r
$
with some $0<\th<1$. It follows with Theorem~\ref{comp-gen} that
\[
  \nn*{h^s\cmp g^{\smash{i}}}^m_{K,r} \le \nn{h^s}^m_{g^i(K),\th^ir}, \qq i\ge0.
\]
As $g^i(K)$ converges to $E^u\cap K$ and $h^s$ is flat on $E^u$ it follows that 
\[
  \nn*{h^s\cmp g^{\smash{i}}}^m_{K,r} = O(\th^{iN})
\]
for any $N\ge1$. Therefore, the first sum converges with respect to~$\nn\cd^m_{K,r}$. An analogous argument applies to the second sum. Consequently, ~\eqref{X-def} defines a \textit{bona fide} solution to equation~\eqref{X} in some neighbourhood of the origin, which is as smooth as any single term in its representation.

As $m$ is supposed to be log-convex, all forward and backward iterates of $\gt$ are of class~$E^m$, as is their composition with $h^s$ and~$h^u$. Moreover, the differential of $\gt$ and its inverse are of class~$E^{\mp}$, as is any product of those. Thus we conclude that the vector field $X$ thus constructed is of class $E^{\mp}$. 

This holds uniformly for each $0\le t\le1$, which we dropped from the notation, while the dependence on $t$ is analytic by construction. Therefore, the time-\m{1}-map $\ph$ of $X$ is a diffeomorphism of class~$E^{\mp}$, which solves the problem of removing the flat term~$h$ near the origin.
\end{proof}

\appendix

\section{Basic facts}  \label{app-1}

For the inconvenience of the reader we collect some basic and well known properties of the spaces~$E^m$ which are determined by the asymptotic behaviour of the weight sequence~$m$. 

\begin{lem} \label{a-prop}
Let 
\mms
$\al=\liminf m_n\rn$.
\enum[z]
\item•1
If\/ $\al=0$, then $C^\om \setm E^m\ne\eset$.
\item•2
If\/ $\al>0$, then $C^\om \subset E^m$, and vice versa.
\item•3
If\/ $\al=\iny$, then $C^\om \subsetneq E^m$.
\endenum
\end{lem}

\begin{proof}
For simplicity we consider functions on some interval around~$0$.

\zg1.
If $\al=0$, then $m_{n_i}^{1/n_i}\to0$ for some subsequence. Then the function 
\mmx
$f = \sum_{i\ge1} x^{n_i}$
\[
  f = \sum_{i\ge1} x^{n_i}
\]
is analytic around~$0$, but 
\[
  \Md_0f = \sum_{i\ge1} \frac{x^{n_i}}{m_{n_i}}
\]
is not. Hence $f$ is an element of $C^\om \setm E^m$.

\zg2.
If $f\in C^\om$, then its power series expansion at any point has a radius of convergence which is locally bounded away from zero, whence
\[
  \limsup_{n\ge1} \n{f_n}\rn \le R < \iny
\]
locally uniformly. Hence also
\[
  \limsup_{n\ge1} \pas{\frac{\n{f_n}}{m_n}}\rn
  \le \frac{\limsup \n{f_n}\rn}{\liminf\, m_n\rn}
  \le \frac{R}{\al}
  < \iny.
\]
So the radius of convergence of $\Md^m_a\f$ is also locally bounded away from zero, whence $f\in E^m$. 
---
Conversely, if $C^\om\subset E^m$, then in particular
\[
  f = \frac{1}{x+\j} = \sum_{n\ge0} \j^{n-1}x^n
\]
is in~$E^m$. We conclude that
\[
  \Md^m_0\f = \sum_{n>0} \frac{x^n}{m_n}
\]
has a positive radius of convergence. This implies that $\al>0$.

\Zg3.
We have $C^\om\subset E^m$ by \ref{2}. On the other hand, by Lemma~\ref{char-func} there exists for any given point a characteristic function $f$ in $E^m$ such that at this point,
\[
  \n{f_n} \ge m_n, \qq n\ge1.
\]
As
\[
  \liminf_{n\ge1} \n{f_n}\rn \ge \liminf_{n\ge1} m_n\rn = \iny,
\]
its Taylor series has no positive radius of convergence, hence $f$ is not analytic.
\end{proof}

Some further properties relate to the asymptotic behaviour of the associated derivative weights $M_n = n\fac m_n$. 

\begin{lem} \label{A-prop}
Let 
\mms 
$A = \liminf M_n\rn$.
\enum[z]
\item•1
If\/ $A>0$, then $E^m \supset E^\om$, the space of entire functions.
\item•2
If\/ $A<\iny$, then $E^m \subset E^\om$.
\item•3
If\/ $A = \iny$, then 
\mms $E^m = E^{\mb}$,
where $\mb$ is the largest weakly log-convex minorant below $m$.
\endenum
\end{lem}

\begin{proof}
\zg1.
If $A>0$, then 
$
  M_n \ge a^n
$
for all $n\ge1$ with some $a>0$, hence
\[
  m_n \ge \frac{a^n}{n\fac}, \qq n\ge1.
\]
It follows that $E^m \supset E^{(1/n\fac)} = E^\om$.
\q
\zg2.
For $f\in E^m$ we locally have
\[
  \nn*{f^{(n)}}_U \le M_nr^n, \qq n\ge1,
\]
with some $r>0$.
If $A<\iny$, then $M_n\le b^n$ for infinitely many~$n$, hence
\[
  \nn*{f^{(n)}}_U \le c^n
\]
for infinitely many $n$ with some $c>0$.
By the Landau-Kolmogorov inequalities~\cite{Che-93} this then also holds for \emph{all~$n$} with some larger~$c$. Thus,
\[
  \nn{f_n}_U \le \frac{c^n}{n\fac}, \qq n\ge1.
\]
It follows that $f\in E^{(1/n\fac)}=E^\om$.

\zg3.
This follows from the Landau-Kolmogorov interpolation inequalities 
\[
  \nn*{D^{\lm p+(1-\lm)q}f} \le c_{p,q,\lm} \nn{D^p\f}^\lm \nn{D^q\f}^{1-\lm},
\]
where $\lm p+(1-\lm)q$ is any integer between $p$ and~$q$, and the fact that $M$ must coincide with $\Mb$ at infinitely many points~\cite{Che-93}.
\end{proof}

\begin{lem} \label{not-comp-closed}
If\/ $A=\liminf M_n\rn<\iny$, then $E^m$ is not stable under composition.
\end{lem}

\begin{proof}
If $A > 0$, then $\exp\in E^m$. But if $A<\iny$, then
$\exp\cmp\exp\notin E^m$, because its \m{n^2}-Taylor coefficients are larger than $(1/n\fac)^{n+1}$, which is not of the order of $1/(n^2)\fac$ \cite{Ban-46}.

If $A=0$, on the other hand, then there are $\ep_n\down0$ such that
\mmx
$M_n\le \ep_n^n$ for all~$n$
\[
  M_n \le \ep_n^n, \qq n\ge0,
\]
with equality holding for infinitely many~$n$. Then
\[
  e = \sum_{n\ge0} \frac{\ep_n^n}{n\fac} x^n \in E^m,
\]
but $e\cmp e\notin E^m$ by a similar calculation. 
\end{proof}

\section{Examples of weights}  \label{a-xmp}

\begin{xmp}  \label{asm-not-fdb}
There are weights, which are \asm, but not \fdb. 
\end{xmp}

Hence there are \m{E^m}-spaces, which are \m{C^\om}-closed, but not \m{E^m}-closed. These are obviously supersets of $C^\om$. This seems to be a new observation.

\begin{proof} \def\M-#1{m_{#1}^{1/#1}} \def\MM-#1{M_{#1}^{1/#1}} 
We first construct an almost increasing sequence $m$ depending on parameters~$\lm_1\le \lm_2\le \bots$. We subsequently choose them so that $m$ is not~\fdb.

Beginnig with $\nb_1=1$ and $\mu_1=1$, we proceed by induction and assume that we already determined $\nb=\nb_{n-1}$ and
$\mu_1\le\mu_2\le\bots\le\mu_{\nb}$ 
such that
\[
  \6
  \M-k \le 8\M-l, \qq 1\le k\le l\le\nb.
\]
We then set
\[
  \mu_k = \lm_nk, \qq \nb<k\le n,
\]
with some $\lm_n\ge 8\M-\nb$ and $n$ so large that
\mms $
  \MM-n \ge \mu_n/4
$.
Subsequently we set
\[
  \mu_k = 2^6\ceil1{{k/n}}\MM-n, \qq n<k\le n^2.
\]

Obviously, $\mu_k$ is increasing for $\nb\le k\le n^2$. Similarly, $\M-k$ is increasing for $\nb\le k\le n$ by choice of~$\lm_n$. Otherwise we observe that 
\[
  M_{in}
  = 2^{6(i-1)n}i\fac^nM_n^i, \qq 2\le i\le n.
\]
As
\[
  \frac{i\fac^n n\fac^i}{(in)\fac}
  = \pas{\frac{\smash{i\fac^{1/i}\,n\fac^{1/n}}}{(in)\fac^{1/in}}}^{in}
  \ge \pas{\frac{1}{2\e}}^{in}
\]
by Stirling's inequality, we conclude that
\[
  m_{in}
  = \frac{M_{in}}{(in)\fac}
  \ge 8^{in} \frac{i\fac^n M_n^i}{(in)\fac} 
  = 8^{in} \frac{i\fac^n n\fac^i}{(in)\fac} m_n^i
  \ge m_n^i,
  \qq
  2\le i\le n,
\]
and therefore
\[
  \M-{n} \le \M-{in}, \qq 1\le i\le n. 
\]
With~\eqref{mkml} this implies that \6 now holds for $1\le k\le l\le n^2$.
Setting $\nb_n=n^2$ this completes the inductive construction of the weight~$m$.

Now consider the \fdb-property. For $k_1=\bots=k_n=n$, we have
\[
  \frac{m_nm_{k_1}\bots m_{k_n}}{m_k}
  =   \frac{m_n^{n+1}}{m_{n^2}}
  =   \frac{(n^2)\fac}{(n\fac)^{n+1}} \frac{M_n^{n+1}}{M_{n^2}}
  \ge   2^{-6n^2}\frac{(n^2)\fac}{(n\fac)^{2n}} \frac{M_n}{n\fac}.
\]
As
\[
  \frac{n^2\fac}{n\fac^{2n}}
  = \pas{\frac{\smash{n^2\fac^{1/n^2}}}{n\fac^{2/n}}}^{n^2}
  \ge \pas{\frac{\e}{4}}^{n^2},
\]
we conclude that
\[
  \frac{m_nm_{k_1}\bots m_{k_n}}{m_k}
  \ge 2^{-7n^2} m_n. 
\]
Choosing the $\lm_n$ and hence the $m_n$ to increase suffciently fast, the right hand side increases faster than any power of $n^2$. Hence, the sequence $m$ is not~\fdb.
\end{proof}

\begin{xmp}  \label{fdb-not-log}
There are weights, which are strictly \fdb, but which are not equivalent to any log-convex weight.
\end{xmp}

Hence log-convexity is not necessary to get stability under composition. 
See also~\cite{RaS-15} for an entirely different example of this kind.

\begin{proof}
If two weights $m$ and $\mt$ are equivalent, then also their successive quotients $\al_n=m_{n}/m_{n-1}$ and $\t\al_n=\mt_{n}/\mt_{n-1}$ form equivalent sequences. If $\mt$ is log-convex, these $\t\al_n$ are increasing, hence for an equivalent weight $\al_{n+1}/\al_n$ can not approach zero faster than exponentially.
But it is easy to construct block-convex weights where this is the case. So these weights are strictly \fdb by Lemma~\ref{prop-fdb}, but not log-convex.
\end{proof}

\begin{xmp}  \label{fdb-not-asm}
There are weights, which are strictly \fdb, but not \asm.
\end{xmp}

The corresponding space $E^m$ is thus a proper subset of $C^\om$, which is stable under composition, but not holomorphically stable.

\begin{proof}
An explicit example is 
\[
  m_n = \log^{-n}(1+n), \qq n\ge1.
\]
It is an elementary task to check that $m$ is weakly log-convex and log-anticonvex. Hence, 
\[
  m_km_l \le m_{k+1}m_{l-1}, \qq 1\le k<l-1.
\]
So condition~\ref{9} of Lemma~\ref{prop-fdb} needs only be checked for $1\le k\le l\le k+1$ -- which is another elementary calculation -- to show that $m$ is strictly \fdb. But $m$ is not \asm, since obviously
\mmx
$\lim m_n\rn = 0$.
\[
  \lim m_n\rn = 0.
  \qedhere
\]
\end{proof}


\begin{xmp}  \label{asm-not-diff}
There are weights, which are \fdb and \asm, but not diff-stable.
\end{xmp}


\begin{proof}
Any log-convex weight $m$ is \fdb and almost increasing, hence \asm. But if the $\mu_n$ increase fast enough so that
\mms$
  \mu_n\rn \to \iny
$,
then $m$ is not closed under differentiation.
\end{proof}

\section{The Cauchy-Kowalewskaya theorem}  \label{s-CK}

Reduced to normal form the problem is to find a solution to
\[
  •ck
  \del_su = c_0(\xs,u) + \sum_{1\le i<s} c_i(\xs,u)\del_iu, 
  \qq
  \eval2{u}_{x_s=0} = 0,
\]
in a neighbourhood of the origin in \m{s}-space, where $\xs=(x_1,\bots,x_{s-1})$,
\[
  u = u(x) = (u_1(x),\bots,u_{\vq}(x)),
\] 
and $c_0$ and $c_1,\bots,c_{s-1}$ are defined in a neighborhood of the origin in \m{s+\vq}-space and take values in \m{\vq}-space and \m{\vq\x\vq}-space, respectively. 

\begin{thm}
Suppose \eqref{ck} has a smooth solution~$u$. If the coefficients $c_0,c_1,\bots,c_{s-1}$ are of strict \fdb class~$E^m$ in some neighborhood of the origin in \m{s}-space, then $u$ is also of class~$E^m$.
\end{thm}

\begin{proof}
We need the following two extensions of the Main Lemma. First we need to consider products of smooth functions. Assuming without loss that $m_l\ge1$ for all $l$ with $\n{l}=2$, we have
\[
  m_{k_1}m_{k_2} \le m_{l}m_{k_1}m_{k_2} \le m_{k_1+k_2}
\]
for a strict \fdb weight. This implies that
$
  M_a(gh) \cle M_ag M_ah
$.
Note that here we have to include the constant terms.

Second, we need to consider `partial composition'.
Suppose $g=g(x,w)$ and $h=h(x)$ are such that
$
  (g\pcmp h)(x) \defeq g(x,h(x))
$
is well defined. We then have
\[
  \Md_a(g\pcmp h) \cle \Md_bg \pcmp \Md_ah, \qq b=(a,h(a)).
\]
This follows from the Main Lemma by extending $h$ to a map which is the identity in the \m{x}-coordinates.

We now proceed as in the proof of Theorem~\ref{ode} and expand both sides of the differential equation~\eqref{ck} into their formal power series at the origin. We get
\[
  \del_s(T_0u)
  = T_0(\del_su)
  = T_0c_0(\xs,u) + \sum_{1\le i<s} T_0(c_i(\xs,u)\del_iu).
\]
Using the Main Lemma and the preceding remarks to pass to their weighted majorants, we get
\[
  \del_s(\Md_0u)
  \cle M_0c_0\pcmp \Md_0u + \sum_{1\le i<s} (M_0c_i\pcmp\Md_0u)(\del_i\Md_0u).
\]
Here, the coefficients $M_0c_0,\bots,M_0c_{s-1}$ are all analytic around the origin by assumption. Hence, by the Cauchy-Kowalewskaya theorem there exists an analytic solution $v$ to
\[
  \del_sv = M_0c_0(\xs,v) + \sum_{1\le i<s} M_0c_i(\xs,v)\del_iv, 
  \qq
  \eval2{v}_{x_s=0} = 0.
\]
By recursive comparison of coefficients we have
$
  \Md_0u \cle \Td_0v
$.
Hence, the smooth solution $u$ is of class $E^m$ near the origin.
\end{proof}

\goodbreak[6]

\end{document}